\documentclass[leqno,11pt]{amsart}

\usepackage{amssymb}
\newcommand{\F}{{\mathbb{F}}}
\newcommand{\Q}{{\mathbb{Q}}}
\newcommand{\Z}{{\mathbb{Z}}}
\newcommand{\K}{{\mathbb{K}}}

\newcommand{\ba}{{\mathbf{a}}}
\newcommand{\bB}{{\mathbf{B}}}
\newcommand{\bG}{{\mathbf{G}}}
\newcommand{\bO}{{\mathbf{O}}}
\newcommand{\bT}{{\mathbf{T}}}
\newcommand{\bW}{{\mathbf{W}}}
\newcommand{\fB}{{\mathfrak{B}}}
\newcommand{\cO}{{\mathcal{O}}}

\newcommand{\cD}{{\mathcal{D}}}
\newcommand{\cF}{{\mathcal{F}}}
\newcommand{\cH}{{\mathcal{H}}}
\newcommand{\cR}{{\mathcal{R}}}
\newcommand{\cS}{{\mathcal{S}}}
\newcommand{\cLR}{{\operatorname{LR}}}
\newcommand{\Irr}{\operatorname{Irr}}
\newcommand{\Unip}{\operatorname{Unip}}

\newtheorem{thm}{Theorem}[section]

\newtheorem{prop}[thm]{Proposition}

\newtheorem{conj}[thm]{Conjecture}

\theoremstyle{definition}

\newtheorem{exmp}[thm]{Example}

\theoremstyle{remark}
\newtheorem{rem}[thm]{Remark}

\renewcommand{\leq}{\leqslant}
\renewcommand{\geq}{\geqslant}

\address{Institute of Mathematics, University of Aberdeen, Aberdeen 
AB24 3UE, Scotland, UK}

\email{m.geck@abdn.ac.uk}

\begin{document}

\date{}

\title{Remarks on modular representations of finite groups of
Lie type in non-defining characteristic}

\author{Meinolf Geck}



\begin{abstract} Let $G$ be a finite group of Lie type and $\ell$ be a 
prime which is not equal to the defining characteristic of $G$. In this
note we discuss some open problems concerning the $\ell$-modular 
irreducible representations of $G$. We also establish a strengthening 
of the results in \cite{myprinc} on the classification of the 
$\ell$-modular principal series representations of $G$.
\end{abstract}

\maketitle

\pagestyle{myheadings}
\markboth{Geck}{Remarks on modular representations}

\section{Introduction} \label{sec0}

Let $G$ be a finite group and $\ell$ be a prime number. Let $K$ be a 
field of characteristic $0$ and assume that $K$ is ``sufficiently large''
(that is, $K$ is a splitting field for $G$ and all its subgroups). Let 
$\cO$ be a discrete valuation ring in $K$, with residue field $k$ of 
characteristic $\ell>0$. Let $\Irr_K(G)$ denote the set of irreducible 
representations of $G$ over $K$ (up to isomorphism) and let $\Irr_k(G)$ 
denote the set of irreducible representations of $G$ over $k$ (up to 
isomorphism). In the setting of Brauer's classical modular representation 
theory (see, for example, Curtis--Reiner \cite[\S 16]{CR2}), we have a 
decomposition map 
\[ d_\cO \colon \cR_0(KG) \rightarrow \cR_0(kG)\]
between the Grothendieck groups of finite-dimensional representations
of $KG$ and $kG$, respectively. Given $\rho \in \Irr_K(G)$ and 
$Y \in \Irr_k(G)$, we denote by $\langle \rho:Y\rangle_\cO$ the
corresponding decomposition number, that is, the multiplicity of
the class of $Y$ in the image of the class of $\rho$ under the map
$d_\cO$. Assuming that $\Irr_K(G)$ is sufficiently well known, the
decomposition numbers $\langle \rho:Y\rangle_\cO$ provide a tool for 
using the available information in characteristic $0$ to derive information 
concerning $\Irr_k(G)$. 

We shall consider the situation where $G$ is a finite group of Lie 
type and $\ell$ is a prime which is not equal to the defining 
characteristic of $G$. The work of Lusztig \cite{LuBook}, \cite{Lu5}
provides a complete picture about the classification and the dimensions 
of the irreducible representations of $G$ over $K$. As far as 
$\ell$-modular representations are concerned, the compatibility 
of $\ell$-blocks with Lusztig series (see \cite{BrMi}, \cite{bs1}) 
suggests that $\Irr_k(G)$ is very closely related to $\Irr_K(G)$, where 
one might hope to quantify the degree of ``closeness'' in terms of 
suitable properties of the decomposition numbers of $G$. (This is in 
contrast to the situation for modular representations in the defining 
characteristic, which tend to be far away from $\Irr_K(G)$; see
Remark~\ref{remell}.) 

More precise information is available for groups of type $A_n$ by the work 
of Fong--Srinivasan \cite{fs1} and Dipper--James \cite{DJ1}. However, for 
groups of other types, much less is known except for special 
characteristics (see, for example, Gruber--Hiss \cite{GrHi}) or groups
of small ranks where explicit computations are possible (see, for 
example, \cite{LaMi}, \cite{OkuWak1}, \cite{OkuWak2}).

In Section~\ref{sec1} we formulate a conjecture concerning the 
classification of the ``unipotent'' modular representations of $G$. 
There is considerable evidence that this conjecture holds in general;
see Remark~\ref{rem2}. In Section~\ref{sec2} we provide a partial proof 
as far as the unipotent modular principal series prepresentations of $G$ 
are concerned; this strengthens the results obtained previously in 
\cite{myprinc}. Finally, in Section~\ref{sec3}, we consider the 
dimensions of irreducible representations and state, as a challenge,
a general ``qualitative'' conjecture for the set of all $\ell$-modular 
representations of $G$. 

Note that the conjectures that we state here are not variations of 
general conjectures on finite groups: Besides their potential immediate
interest, they express properties of modular representations of finite 
groups of Lie type which, if true, would provide further evidence for the 
sharp distinction between the non-defining and the defining characteristic 
case. 
 
To fix some notation, let $p$ be a prime number and $\overline{\F}_p$ 
be an algebraic closure of $\F_p=\Z/p\Z$. Let $\bG$ be a connected 
reductive algebraic group over $\overline{\F}_p$ and $F \colon \bG 
\rightarrow \bG$ be a homomorphism of algebraic groups such that some 
power of $F$ is the Frobenius map relative to a rational structure on 
$\bG$ over some finite subfield of $\overline{\F}_p$. Then $\bG^F:=
\{g\in\bG\mid F(g)=g\}$ is called a finite group of Lie type; we shall 
write $G:=\bG^F$. Similar conventions apply to $F$-stable subgroups 
of $\bG$.  Let $\bB \subseteq \bG$ be an $F$-stable Borel subgroup and 
$\bT_0\subseteq \bB$ be an $F$-stable maximal torus. Then we write
$B:=\bB^F$ and $T_0:=\bT_0^F$. Let $\bW=N_{\bG}(\bT_0)/\bT_0$ be the 
Weyl group of $\bG$. Then $F$ induces an automorphism $\gamma \colon 
\bW \rightarrow \bW$. 

Let $\delta\geq 1$ be minimal such that $F^\delta$ is the Frobenius
map relative to a rational structure on $\bG$ over a finite subfield
$k_0 \subseteq \overline{\F}_p$. Define $q>0$ to be the unique real
number such that $|k_0|=q^\delta$. (If $\bG$ is simple modulo its
center, then $\delta=1$ and $q$ is a power of $p$, except when $G$
is a Suzuki or Ree group in which case $\delta=2$ and $q$ is an odd
power of $\sqrt{2}$ or $\sqrt{3}$.)

Throughout this paper (except for the final Remark~\ref{remell}), we 
assume that $K,\cO,k$ as above are such that $\ell\neq p$.

\section{On modular unipotent representations} \label{sec1}

Let $\Unip_K(G)\subseteq \Irr_K(G)$ be the set of unipotent 
representations of $G$, as defined by Deligne--Lusztig \cite{DeLu}. 
(Note that, in order to define $\Unip_K(G)$, one first needs to 
work over $\overline{\Q}_\ell$, where $\ell \neq p$, in order to 
construct the virtual representations $R_{T,1}$ of \cite{DeLu}; 
since the character values of $R_{T,1}$ are rational integers, 
the set $\Unip_K(G)$ is then unambiguously defined for any $K$ of 
characteristic $0$.) The {\em $\ell$-modular unipotent representations}
of $G$ are defined to be
\[ \Unip_k(G):=\{Y\in\Irr_k(G)\mid\langle \rho:Y\rangle_\cO 
\neq 0 \mbox{ for some } \rho \in \Unip_K(G)\}.\]
We wish to state a conjecture about the classification of
$\Unip_k(G)$. First, we recall some results about the characteristic 
$0$ representations of $G$. 

Let $\bO$ be an $F$-stable unipotent conjugacy class of $\bG$. 
Let $u_1\ldots,u_r \in \bO^F$ be representatives of the $G$-conjugacy 
classes contained in $\bO^F$. For each $j$, let $A(u_j)$ be the 
group of connected components of the centraliser of $u_j$ in $\bG$. 
Since $F(u_j)=u_j$, there is an induced action of $F$ on $A(u_j)$ 
which we denote by the same symbol. Now let $\rho \in \Irr_K(G)$.
Then we define the {\em average value} of $\rho$ on $\bO^F$ by 
\[\mbox{AV}(\bO,\rho):=\sum_{1\leq j\leq r}[A(u_j):A(u_j)^F]\,
\mbox{trace}(u_j, \rho).\]
(Note that $\mbox{AV}(\bO,\rho)$ does not depend on the choice of
the representatives $u_j$.) Assuming that $p,q$ are large enough, Lusztig 
\cite{Lu6} has shown that, given $\rho\in \Irr_K(G)$, there exists a 
{\em unique} $F$-stable unipotent class $\bO_\rho$ satisfying the
following two conditions:
\begin{itemize}
\item $\mbox{AV}(\bO_\rho,\rho)\neq 0$ and 
\item if $\bO$ is any $F$-stable unipotent class $\bO$ such 
that $\mbox{AV}(\bO, \rho) \neq 0$, then $\bO=\bO_\rho$ or 
$\dim \bO <\dim \bO_\rho$.
\end{itemize}
The class $\bO_\rho$ is called the {\em unipotent support} of $\rho$. 
The assumptions on $p,q$ have subsequently been removed in \cite{GeMa}. 
Thus, every $\rho\in\Irr_K(G)$ has a well-defined unipotent support 
$\bO_\rho$. Using this concept, we can associate to every $\rho \in 
\Irr_K(G)$ a numerical invariant $\ba_\rho$ by setting
\[ \ba_\rho:=\dim \fB_u \qquad (u \in \bO_\rho)\]
where $\fB_u$ is the variety of Borel subgroups of $\bG$ containing~$u$.

Now consider the unipotent representations $\Unip_K(G)$. By \cite[Main 
Theorem 4.23]{LuBook}, there is a bijection 
\[ \bar{X}(\bW,\gamma) \quad \stackrel{1{-}1}{\longleftrightarrow}
\quad \Unip_K(G), \qquad x \leftrightarrow \rho_x,\]
where $\bar{X}(\bW,\gamma)$ is a finite set depending only on 
the Weyl group $\bW$ of $\bG$ and the automorphism $\gamma \colon
\bW \rightarrow \bW$ induced by the action of $F$. (This bijection
satisfies further properties as specified in \cite[4.23]{LuBook}; 
we shall not need to discuss these properties here.)

We shall need two further pieces of notation. Let $Z_{\bG}$ be the 
center of $\bG$. Then $(Z_\bG/Z_{\bG}^{\circ})_F$ denotes the largest 
quotient of $Z_\bG/Z_{\bG}^\circ$ on which $F$ acts trivially. Also 
recall (e.g., from \cite[p.~28]{C2}) that a prime number is called 
{\em good} for $\bG$ if it is good for each simple factor involved 
in~$\bG$; the conditions for the various simple types are as follows.
\begin{center} $\begin{array}{rl} A_n:\quad& \mbox{no condition}, \\
B_n, C_n, D_n: \quad & \ell \neq 2, \\
G_2, F_4, E_6, E_7: \quad &  \ell \neq 2,3, \\
E_8: \quad & \ell \neq 2,3,5.  \end{array}$
\end{center}
Now we can state:

\begin{conj}[Geck \protect{\cite[\S 2.5]{myphd}}, Geck--Hiss
\protect{\cite[\S 3]{lymgh}}] \label{modconj} Assume that $\ell$ 
is good for $\bG$ and that $\ell$ does not divide the order of 
$(Z_\bG/Z_{\bG}^{\circ})_F$. Then there is a labelling 
$\Unip_k(G)=\{ Y_x \mid x \in \bar{X}(\bW,\gamma)\}$ such that 
the following conditions hold for all $x,x' \in \bar{X}(\bW,\gamma)$:
\begin{alignat*}{2}
\langle \rho_x &: Y_x \rangle_{\cO} &\;=\;&1, \\ \langle \rho_{x'} 
&: Y_x \rangle_{\cO} &\;\neq\; & 0 \quad \Rightarrow \quad
x=x' \quad \mbox{or} \quad \bO_{\rho_{x'}} \subsetneqq 
\overline{\bO}_{\rho_x}.
\end{alignat*}
(Note that, if such a labelling exists, then it is uniquely determined.)
\end{conj}

\begin{rem} \label{rem0} Under the above assumption on $\ell$, it
is known by \cite{bs1}, \cite{bs2} that 
\[ |\Unip_k(G)|=|\Unip_K(G)|.\]
If $\ell$ is not good for $\bG$, or if $\ell$ divides the order of 
$(Z_{\bG}/Z_{\bG}^\circ)_F$, then we have $|\Unip_k(G)|\neq 
|\Unip_K(G)|$ in general. For further information on the cardinalities 
$|\Unip_k(G)|$ in such cases, see \cite[6.6]{lymgh}. 
\end{rem}

\begin{rem} \label{rem1} The formulation of the above conjecture 
is somewhat stronger than that in \cite[Conj.~1.3]{myprinc} where,
instead of the ``geometric'' condition 
\[ \langle \rho_{x'} : Y_x \rangle_{\cO} \;\neq\;  0 \quad \Rightarrow 
\quad x=x' \quad \mbox{or} \quad \bO_{\rho_{x'}} \subsetneqq 
\overline{\bO}_{\rho_{x}},\]
we used the purely numerical condition:
\[ \langle \rho_{x'} : Y_x \rangle_{\cO} \;\neq\; 0 \quad \Rightarrow 
\quad x=x' \quad \mbox{or} \quad \ba_{\rho_{x'}}>\ba_{\rho_{x}}.\]
The stronger version is known to hold for $G=\mbox{GL}_n(\F_q)$ (see
Dipper--James \cite{DJ1} and the references there) and $G=\mbox{GU}_n
(\F_q)$ (see \cite[\S 2.5]{myphd}, \cite[\S 2.5]{myprinc}). The argument
for $\mbox{GU}_n(\F_q)$ essentially relies on Kawanaka's theory 
\cite{Kaw2} of generalised Gelfand--Graev representations; see also 
Remark~\ref{rem2} below. Further support will be provided by 
Proposition~\ref{prop1} below.
\end{rem}

\begin{rem} \label{rem2} Using Brauer reciprocity, 
Conjecture~\ref{modconj} can be alternatively stated as follows. 
There should exist finitely generated, projective $\cO G$-modules $\{\Phi_x
\mid x \in \bar{X}(\bW,\gamma)\}$ such that, for any $x \in
\bar{X}(\bW,\gamma)$, we have:
\begin{align*}
 K \otimes_\cO \Phi_x \cong \rho_x &\oplus \mbox{ (direct sum
of various $\rho_{x'}$ where $\bO_{\rho_{x'}} \subsetneqq 
\bO_{\rho_x}$)}\\
& \oplus \mbox{ (direct sum of various non-unipotent $\rho'\in 
\Irr_K(G)$)}.
\end{align*}
This formulation is particularly useful in connection with
Kawanaka's theory \cite{Kaw2} of generalised Gelfand--Graev 
representations. Assume that $p,q$ are sufficiently large such that
Lusztig's results \cite{Lu6} hold. Let $u \in G$ be a unipotent 
element and denote by $\Gamma_u$ the corresponding generalised 
Gelfand--Graev representation of $G$ over $K$. Since $\Gamma_u$ is 
obtained by inducing a representation from a unipotent subgroup 
of $G$, we have $\Gamma_u \cong K\otimes_{\cO} \Upsilon_u$ where 
$\Upsilon_u$ is a finitely generated, projective $\cO G$-module. 

Now let $x\in \bar{X}(\bW,\gamma)$. Then we can find a unipotent
element $u\in G$ such that $\Gamma_u$ is a linear combination of 
various $\rho'\in \Irr_K(G)$ where $\bO_{\rho'}\subseteq 
\overline{\bO}_{\rho_x}$, and where the coefficient of $\rho_x$ is 
non-zero.  This follows from the multiplicity formula in 
\cite[Theorem~11.2]{Lu6}, together with the refinement obtained by 
Achar--Aubert \cite[Theorem~9.1]{acau}. Thus, $\Upsilon_u$ is a first 
approximisation to the hypothetical projective $\cO G$-module 
$\Phi_x$, more precisely, $\Phi_x$ should be a direct summand of
$\Upsilon_u$. This also shows that the closure relation among unipotent 
supports naturally appears in this context. 

The special feature of the case where $G=\mbox{GL}_n(\F_q)$ or 
$\mbox{GU}_n(\F_q)$ is that then $|\bar{X}(\bW,\gamma)|$ is equal to
the number of unipotent classes of $\bG$ and, using the above
notation, we can just take $\Phi_x$ to be $\Upsilon_u$.

In general, it seems to be necessary to work with certain modified
generalised Gelfand--Graev representations, as proposed by Kawanaka 
\cite[\S 2]{Kaw2}. In this context, Conjecture~\ref{modconj}
would follow from the conjecture in \cite[(2.4.5)]{Kaw2}.
\end{rem}

Finally, we give an alternative description of the closure relation
among unipotent supports which will play a crucial role in 
Section~\ref{sec2}.

\begin{rem} \label{2kl} Let $\leq_{\cLR}$ be the two-sided 
Kazhdan--Lusztig pre-order relation on $\bW$; see
\cite[Chap.~5]{LuBook}. Given $w,w' \in \bW$ we write $w \sim_{\cLR} 
w'$ if $w \leq_{\cLR} w'$ and $w' \leq_{\cLR} w$. The equivalence 
classes for this relation are called the two-sided cells of $\bW$. 
By \cite[5.15]{LuBook}, we have a natural partition
\[ \Irr_K(\bW)=\coprod_\cF \Irr_K(\bW\mid \cF)\]
where $\cF$ runs over the two-sided cells of $\bW$. Now the
automorphism $\gamma \colon \bW \rightarrow \bW$ induces a permutation
of the two-sided cells of $\bW$. By \cite[6.17]{LuBook}, we also have
a natural partition
\[ \Unip_K(G)=\coprod_\cF \Unip_K(G\mid \cF)\]
where $\cF$ runs over the $\gamma$-stable two-sided cells of $\bW$. 
Given $x \in \bar{X}(\bW,\gamma)$, we denote by $\cF_x$ the unique 
$\gamma$-stable two-sided cell of $\bW$ such that  $\rho_x \in 
\Unip_K(G\mid \cF_x)$. By \cite{Lu6} and \cite[Prop.~4.2]{GeMa}, the
above partition can be characterised in terms of unipotent supports as 
follows: 
\[ \cF_x=\cF_{x'} \qquad \Leftrightarrow \qquad \bO_{\rho_x}
=\bO_{\rho_{x'}}\qquad\mbox{for all $x,x'\in\bar{X}(\bW,\gamma)$}.\]
Now $\leq_{\cLR}$ induces a partial order relation on the set of
two-sided cells which we denote by the same symbol. Thus, given 
two-sided cells $\cF,\cF$ of $\bW$, we write $\cF \leq_{\cLR} \cF'$ 
if and only if $w \leq_{\cLR} w'$ for all $w\in\cF$ and $w'\in\cF'$. 
\end{rem}

\begin{prop}[See \protect{\cite[Cor.~5.6]{myord}}] \label{klord} In
the above setting, we have 
\[ \cF_x\leq_{\cLR} \cF_{x'} \qquad \Leftrightarrow \qquad \bO_{\rho_x}
\subseteq \overline{\bO}_{\rho_{x'}}\qquad\mbox{for all $x,x'
\in\bar{X}(\bW,\gamma)$}.\]
\end{prop}

Note that, in \cite{myord}, we work in a slightly different 
setting where two-sided cells and unipotent classes are linked
via the Springer correspondence. However, by \cite{Lu6} and 
\cite[Theorem~3.7]{GeMa}, the map which assigns to each $\rho
\in \Unip_K(G)$ its support support $\bO_\rho$ can also be
interpreted in terms of the Springer correspondence. Thus, indeed,
Proposition~\ref{klord} as formulated above is equivalent to 
\cite[Cor.~5.6]{myord}.

\section{Principal series representations} \label{sec2}
Recall that $\bB \subseteq \bG$ is an $F$-stable Borel subgroup. Let 
$B:=\bB^F \subseteq G$ and consider the permutation module $K[G/B]$ on the 
cosets of $B$. Then the set $\Irr_K(G\mid B)$ of (unipotent) principal
series representations is defined to be the set of all $\rho \in \Irr_K(G)$ 
such that $\rho$ is an irreducible constituent of $K[G/B]$. The modular 
analogue of $\Irr_K(G\mid B)$ is more subtle since $k[G/B]$ is not 
semisimple in general. Following Dipper \cite{QuotHom}, we define 
$\Irr_k(G\mid B)$ to be the set of all $Y \in \Irr_k(G)$ such that 
$\mbox{Hom}_{kG}(k[G/B],Y)\neq \{0\}$. By Frobenius reciprocity, we have
$Y \in \Irr_k(G\mid B)$ if and only if $Y$ admits non-zero vectors 
fixed by all elements of $B$. This fits with a general definition
of Harish--Chandra series for $\Irr_k(G)$; see Hiss \cite{Hiss1}. 
We have 
\[ \Irr_K(G\mid B) \subseteq \Unip_K(G) \quad \mbox{and}\quad  
\Irr_k(G\mid B) \subseteq \Unip_k(G).\] 
Thus, there is a subset $\Lambda \subseteq \bar{X}(\bW,\gamma)$ such
that $\Irr_K(G\mid B)=\{\rho_x \mid x \in \Lambda\}$. This subset
is explicitly described by \cite[Chap.~4]{LuBook}. It is also known 
that this subset is in bijection with $\Irr_K(\bW^\gamma)$; see 
\cite[\S 68B]{CR2}. 

Now assume that $\ell$ is good for $\bG$. Then, by 
\cite[Theorem~1.1]{myprinc} (see also \cite[\S 4.4]{GeNico}), there 
exists a unique subset $\Lambda_k^\circ\subseteq \bar{X}(\bW,\gamma)$ 
and a unique labelling $\Irr_k(G\mid B)=\{Y_x \mid x \in
\Lambda_k^\circ\}$ such that the following conditions hold for all 
$x \in \Lambda_k^\circ$ and $x' \in \bar{X}(\bW, \gamma)$:
\begin{alignat*}{2}
\langle \rho_{x} &: Y_{x}\rangle_{\cO} &\;=\;&1, \\ \langle \rho_{x'} 
&: Y_x \rangle_{\cO} &\;\neq\; & 0 \quad \Rightarrow \quad
x=x' \quad \mbox{or} \quad \ba_{\rho_{x'}}> \ba_{\rho_{x}}.
\end{alignat*}
Furthermore, we have in fact $\Lambda_k^\circ \subseteq \Lambda$. So, 
here, we used the numerical condition in Remark~\ref{rem1}. (This  
was the only condition available at the time of writing \cite{myprinc}.)
Our aim now is to show that we can replace this condition by the 
condition involving the closure relation among unipotent supports. 

\begin{prop} \label{prop1} Assume that $F$ acts as the identity on $\bW$
(that is, $\gamma=\operatorname{id}$). Then, using the above notation, 
the following implication holds for all $x\in \Lambda_k^\circ$ and 
$x'\in \bar{X}(\bW,\gamma)$:
\[\langle \rho_{x'} : Y_x \rangle_{\cO} \;\neq\;  0 \quad \Rightarrow 
\quad x=x' \quad \mbox{or} \quad \bO_{\rho_{x'}} \subsetneqq 
\overline{\bO}_{\rho_x}.\]
\end{prop}

\begin{proof} We go through the main steps of the proof of 
\cite[Theorem~1.1]{myprinc}. For this purpose, consider the Hecke 
algebra $\cH_\cO=\mbox{End}_{\cO G} (\cO[G/B])$ and let $\cH_K:=K 
\otimes_{\cO}  \cH_{\cO}$ and $\cH_k:=k \otimes_{\cO}  \cH_{\cO}$. 
Since $\gamma=\mbox{id}$, the algebra $\cH_\cO$ has a standard basis
$\{T_w \mid w \in \bW\}$ where the multiplication is given as follows,
where $s \in \bW$ is a simple reflection, $w \in \bW$ and $l$ is 
the length function:
\[ T_sT_w=\left\{\begin{array}{cl} T_{sw} & \qquad 
\mbox{if $l(sw)=l(w)+1$},\\ qT_{sw}+(q-1)T_w & \qquad 
\mbox{if $l(sw)=l(w)-1$}.\end{array}\right.\]
We also have a decomposition map between the Grothendieck groups 
of $\cH_K$ and $\cH_k$. Given $E \in \Irr(\cH_K)$ and $M \in 
\Irr(\cH_k)$, denote by $d_{E,M}$ the corresponding decomposition 
number. Now, by results of Dipper \cite{QuotHom} (see 
\cite[\S 2.2]{myprinc}), we have natural bijections 
\begin{alignat*}{2}
\Irr(\cH_K) \quad &\stackrel{1{-}1}{\longrightarrow} \quad
&\Irr_K(G\mid B), \qquad E &\mapsto \rho_E,\\
\Irr(\cH_k) \quad &\stackrel{1{-}1}{\longrightarrow} \quad
&\Irr_k(G\mid B), \qquad M &\mapsto Y_M;
\end{alignat*}
furthermore, for any $x \in \Lambda_k^\circ$ and $x' \in 
\bar{X}(\bW,\gamma)$, we have 
\begin{equation*}
\langle \rho_{x'}:Y_x\rangle_{\cO}=\left\{\begin{array}{cl}
d_{E,M} & \quad \mbox{if $\rho_{x'} \cong \rho_E$ and $Y_x \cong Y_M$},\\
0 & \quad \mbox{otherwise}.\end{array}\right.\tag{$*$}
\end{equation*}
Thus, we are reduced to a problem within the representation theory of 
$\cH_\cO$. Now, with every $E \in \Irr(\cH_K)$ we can associate a 
two-sided Kazhdan--Lusztig cell $\cF_E$ of $\bW$; see \cite[5.15]{LuBook}. 
As in \cite[\S 2.3]{myprinc}, let us define $\ba_E:=\ba_{\rho_E}$.
Then, by \cite[\S 2.4]{myprinc} (see also \cite[Prop.~3.2.7]{GeNico}),
there exists a unique injection 
\[ \Irr(\cH_k) \hookrightarrow \Irr(\cH_K), \qquad M \mapsto E_M,\]
satisfying the following conditions for all $M \in \Irr(\cH_k)$ and
$E \in \Irr(\cH_K)$:
\begin{align*}
d_{E_M,M}&=1, \tag{a}\\
d_{E,M} &\neq 0 \quad \Rightarrow \quad E\cong E_M \quad \mbox{or} 
\quad \ba_E>\ba_{E_M}.\tag{b}
\end{align*}
Now we can argue as follows. Assume that $\langle \rho_{x'}: Y_{x} 
\rangle_{\cO}\neq 0$ where $x\in \Lambda_k^\circ$ and $x'\in \bar{X}
(\bW,\gamma)$. By ($*$), we must have $\rho_{x'}\cong \rho_E$
for some $E \in \Irr(\cH_K)$ and $Y_x \cong Y_M$ for some $M \in 
\Irr(\cH_k)$; this also shows that $\rho_x\cong \rho_{E_M}$. 
Furthermore, $d_{E,M}=\langle \rho_{x'}:Y_x\rangle_{\cO} \neq 0$ and so, 
using (b), we obtain:
\[ \langle\rho_{x'}:Y_x\rangle_{\cO} \neq 0 \quad \Rightarrow \quad x=x' 
\quad \mbox{or} \quad \ba_{\rho_{x'}}>\ba_{\rho_{x}},\]
as in the original version of \cite[Theorem~1.1]{myprinc}. In order to 
prove the strengthening, we use the results in \cite{mycell}, 
\cite{myedin} which show that the following variation of (b) holds:
\[ d_{E,M} \neq 0 \quad \Rightarrow \quad E\cong E_M \quad \mbox{or}
\quad \cF_{E} \leq_{\cLR} \cF_{E_M},\; \cF_E \neq \cF_{E_M}.\]
Hence, arguing as above, we obtain
\[ \langle\rho_{x'}:Y_x\rangle_{\cO} \neq 0 \quad \Rightarrow \quad x=x' 
\quad \mbox{or} \quad \cF_{E} \leq_{\cLR} \cF_{E_M}, \; \cF_E \neq
\cF_{E_M}.\]
Now, the multiplicity formula in \cite[Main Theorem 4.23]{LuBook} shows
that $\rho_E$ appears with non-zero multiplicity in the 
``almost-representation'' $R_{E_0}$ of $G$ associated with some $E_0 \in 
\Irr(\cH_K)$ where $\cF_E=\cF_{E_0}$. Then \cite[6.17]{LuBook} shows
that $\cF_E=\cF_{E_0}=\cF_{\rho_E}=\cF_{x'}$. Similarly, since 
$\rho_x \cong \rho_{E_M}$, we have $\cF_{E_M}=\cF_{x}$. Hence, it remains 
to use the equivalence in Proposition~\ref{klord}.
\end{proof}

\section{On the dimensions of irreducible representations} \label{sec3}
Finally, we wish to state a conjecture on the dimensions
of the irreducible representations of $G$ over $k$. Let us first 
consider the situation in characteristic~$0$. Then 
\cite[Main Theorem~4.23]{LuBook} implies that there exists a collection 
of polynomials $\{D_x \mid x \in \bar{X}(\bW,\gamma)\} \subseteq 
\K[t]$ (where $t$ is an indeterminate and $\K$ is a finite
extension of $\Q$ of degree $[K:\Q]=\delta$) such that 
\[ \dim \rho_x=D_x(q) \qquad \mbox{for all $x \in \bar{X}
(\bW,\gamma)$}.\] 
This set of polynomials only depends on $\bW$ and $\gamma$.
Using Lusztig's Jordan decomposition of representations and 
\cite[Prop.~5.1]{Lu5}, one can in fact formulate a global statement 
for all $\Irr_K(G)$, as follows. 

\begin{prop}[Lusztig \protect{\cite{LuBook}, \cite{Lu5}}] 
\label{propdim} There exists a finite set of polynomials 
$\cD_0(\bW,\gamma) \subseteq \K[t]$, depending only on $\bW$ and
$\gamma$, such that 
\[ \{\dim \rho \mid \rho \in \Irr_K(G)\} \quad \subseteq \quad 
\{f(q) \mid f \in \cD_0(\bW,\gamma)\}.\]
\end{prop}

For example, if $\bW$ is of type $A_1$ and $\gamma=\mbox{id}$ (where,
for example, $G=\mbox{GL}_2(\F_q)$ or $\mbox{SL}_2(\F_q)$ and $q$ is 
any prime power), then we can take
\[\cD_0(A_1,\mbox{id})=\Big\{1,t,t\pm 1,\frac{1}{2}(t\pm 1)\Big\}.\]
If $\bW$ is of type $C_2$ and $\gamma\neq \mbox{id}$ (where 
$G$ is a Suzuki group and $q$ is an odd power of $\sqrt{2}$), then 
we can take
\[\cD_0(C_2,\gamma)=\Big\{1,t^4,t^4+1,\frac{1}{\sqrt{2}}t(t^2-1),
(t^2-1)(t^2\pm t\sqrt{2}+1)\Big\}.\]

\begin{rem} \label{rem3} The results in \cite{LuBook}, \cite{Lu5}
yield a precise and complete description of a set of polynomials
which are needed to express $\dim \rho$ for all $\rho \in \Irr_K(G)$. 
However, here we will only be interested in a qualitative statement
where it will be sufficient to find some, possibly much too large,
but still {\em finite} set of polynomials $\cD_0(\bW,\gamma)$ by which 
we can express $\dim \rho$ for all $\rho \in \Irr_K(G)$.

Note that it is actually not difficult to find such a set $\cD_0(\bW,
\gamma)$. Indeed, for any $w \in \bW$, let $\bT_w \subseteq \bG$ be an
$F$-stable maximal torus of type $w$ and denote by $R_{\bT_w,\theta}$ the
virtual representation defined by Deligne and Lusztig \cite{DeLu}, where 
$\theta \in \Irr_K(\bT_w^F)$. By \cite[\S 2.9, 3.3.8, 7.5.2]{C2}, there 
exists a polynomial $f_w \in \Z[t]$ (depending only on $\bW, \gamma$ and 
$w$) such that $\dim R_{\bT_w,\theta}=f_w(q)$. Then the set 
\[ \cD_0(\bW,\gamma):=\Bigl\{\sum_{w \in \bW} \frac{a_w}{b_w}\, f_w
\;\Big|\; \begin{array}{l} a_w\in \Z \mbox{ and } b_w \in \Z 
\mbox{ such that } \\ |a_w|\leq |\bW| \mbox{ and } 0<|b_w|\leq 
|\bW|\end{array}\Bigr\}\]
has the required properties. This follows using the scalar product 
formula for $R_{\bT_w,\theta}$, the partition of $\Irr_K(G)$ into 
geometric conjugacy classes, and the uniformity of the regular 
representation of $G$. These results can be found in \cite[7.3.4, 
7.3.8, 7.5.6]{C2}; note that these were all already available by 
\cite{DeLu}.
\end{rem}

Now consider the situation in characteristic $\ell>0$.

\begin{conj} \label{dimconj} There exists a finite set of polynomials
$\bar{\cD}(\bW,\gamma)\subseteq \K[t]$, depending only on $\bW,\gamma$
(but not on $q$ or $\ell$), such that 
\[ \{ \dim Y \mid Y \in \Irr_k(G) \} \quad \subseteq \quad \{f(q) 
\mid f \in \bar{\cD}(\bW,\gamma)\}.\]
\end{conj}

\begin{rem} \label{rem4} If this conjecture holds then, in particular,
the set $\bar{\cD}(\bW,\gamma)$ will satisfy the condition in
Proposition~\ref{propdim}. (To see this just choose $\ell$ such 
that $\ell\nmid |G|$.) Thus, we may always assume that ${\cD}_0
(\bW,\gamma) \subseteq \bar{\cD}(\bW,\gamma)$. This inclusion will be 
strict in general. For example, if $\bW$ is of type $C_2$ and $\gamma$ 
is non-trivial, then the results in \cite{Burk} show that 
\[ \bar{\cD}(C_2,\gamma)=\cD_0(C_2,\gamma) \cup \{t^4-1\}.\] 
The point of the conjecture is that, in general, it should be sufficient 
to add only finitely many polynomials to $\cD_0(\bW,\gamma)$ in order to 
obtain $\bar{\cD}(\bW,\gamma)$.
\end{rem}

\begin{exmp} \label{bgl} Conjecture~\ref{dimconj} is true for $G=
\mbox{GL}_n(\F_q)$. This seen as follows. Let $D$ be the $\ell$-modular 
decomposition matrix of $G$. Recall that this matrix has rows labelled by 
$\Irr_K(G)$ and columns labelled by $\Irr_k(G)$. By Fong--Srinivasan 
\cite{fs1}, there is a subset $\cS \subseteq \Irr_K(G)$ such that the 
submatrix of $D$ with rows labelled by $\cS$ is square and invertible 
over $\Z$. Let $D_0$ denote this submatrix. Then we obtain each $\dim Y$ 
(where $Y \in \Irr_k(G)$) as an integral linear combination of $\{\dim 
\rho \mid \rho \in\cS\}$ where the coefficients are entries of the 
inverse of $D_0$. The results in \cite{fs1} show, furthermore, that 
$D_0$ is a block diagonal matrix where the sizes of the diagonal blocks 
are bounded by a constant which only depends on $n$ (but not on $q$ or 
$\ell$). Now, Dipper--James \cite{DJ1} showed that these diagonal blocks 
of $D_0$ are given by the decomposition matrices of various $q$-Schur 
algebras. Each of these algebras is finite-dimensional where the dimension 
is bounded by a function in $n$. Hence, $D_0$ is a block diagonal matrix 
where both the sizes and the entries of the diagonal blocks are uniformly 
bounded by a constant which only depends on $n$. (But note that the total
size of $D_0$ depends, of course, on $q$ and $\ell$.) Analogous statements
then also hold for the inverse of $D_0$, with the only difference that the 
entries may be negative (but the absolute values will still be bounded by 
a constant which only depends on~$n$). We conclude that each $\dim Y$ 
(where $Y\in\Irr_k(G)$) can be expressed as an integral linear combination 
of $\{f(q) \mid f \in \cD_0(\bW,\gamma)\}$ where the absolute values of 
the coefficients are bounded by a constant which only depends on $n$. By 
taking all possible such linear combinations of the polynomials in 
$\cD_0(\bW,\gamma)$, we obtain a finite set $\bar{\cD}(\bW,\gamma)$ with 
the required property. 
\end{exmp}

We note that no further examples are known except for some types of 
groups of small rank where explicit computations are possible and one 
can adopt the above arguments; see \cite{LaMi}, \cite{OkuWak1}, 
\cite{OkuWak2} and the references there. In particular, the problem is 
open for the groups $G=\mbox{GU}_n(\F_q)$.

\begin{rem} \label{remell} Recall that, throughout this paper, we have
assumed that $\mbox{char}(k)=\ell\neq p$. In this final remark, we drop 
this asumption and let $\cO$ be such that $\mbox{char}(k)=\ell=p$. Consider
the example $G=\mbox{SL}_2(\F_p)$ where $\bW$  is a cyclic group of order 
$2$. Then 
\[\{\dim Y \mid Y \in \Irr_k(G)\}=\{1,2,\ldots,p\} \qquad \mbox{(see 
\cite[\S 3]{Alp})}.\]
So it is impossible that a statement like that in Conjecture~\ref{dimconj} 
holds for $\Irr_k(G)$ where $\mbox{char}(k)=\ell=p$. -- Thus, 
Conjecture~\ref{dimconj} is an indication of the sharp distinction between 
the modular representation theory of finite groups of Lie type in defining 
and non-defining characteristic.
\end{rem}

 
\end{document}